\documentclass[a4paper,oneside,english]{amsart}
\usepackage{amstext}
\makeatletter

\usepackage{amsmath,amsthm}
\usepackage{amssymb}
\usepackage{graphics}
\usepackage{epsfig}
\usepackage{psfrag}
\usepackage[T1]{fontenc}
\usepackage[utf8]{inputenc}

\advance \topmargin by -\headheight
\advance \topmargin by -\headsep
\evensidemargin \oddsidemargin
\newtheorem{theorem}{Theorem}[section]
\newtheorem{lemma}[theorem]{Lemma}
\newtheorem{prob}[theorem]{Problem}

\newtheorem{conjecture}[theorem]{Conjecture}

\newtheorem{corollary}[theorem]{Corollary}

\newtheorem{claim}[theorem]{Claim}

\def\endpf{\hfill\rule{2mm}{2mm}\\\bigskip}

  \theoremstyle{definition}
  \newtheorem{definition}[theorem]{Definition}
  \theoremstyle{remark}
  
  \newtheorem{observation}[theorem]{Observation}
  \newtheorem{example}[theorem]{Example}
  \theoremstyle{plain}

\newcommand{\cf}{{\mathcal F}}
\newcommand{\cm}{{\mathcal M}}
\newcommand{\cn}{{\mathcal N}}
\newcommand{\cc}{{\mathcal C}}
\newcommand{\ci}{{\mathcal I}}

\@ifundefined{definecolor}
 {\usepackage{color}}{}
\makeatother

\usepackage{babel}

\title{On a generalization of the Ryser-Brualdi-Stein conjecture}
\begin{document}

\makeatother

\author{Ron Aharoni}\thanks{The research of Ron Aharoni was supported by ISF grant number $2017006$ , by BSF grant number $2006099$,  by ARC grant number DP120100197  and by the Discount Band chair.}

\address{Department of Mathematics\\
 Technion, Haifa\\
 Israel 32000}

\email{Ron Aharoni: ra@tx.technion.ac.il}

\author{Pierre Charbit}

\address{ LIAFA, Universit\'e Paris Diderot}

\email{Pierre Charbit: charbit@liafa.jussieu.fr}

\author{David Howard}

\address{Department of Mathematics, Colgate University, Hamilton, NY 13346}

\email{David Howard: dmhoward@colgate.edu}

\maketitle

\begin{abstract}
A {\em rainbow matching} for (not necessarily distinct) sets $F_1,\ldots ,F_k$ of hypergraph edges
 is a matching  consisting of $k$ edges, one from each
$F_i$.  The aim of the paper is twofold - to
put order in the multitude of conjectures that relate to this concept (some of them first presented here), and to present some partial results on one of these conjectures, that seems central among them.
\end{abstract}

\section{Introduction }
A choice function for a family of sets  $\cf=(F_1, \ldots ,F_m)$ is a choice of elements $f_1 \in F_1, \ldots, f_m \in F_m$.  It is also called a {\em system of representatives} (SR) for $\cf$. Many combinatorial questions can be formulated in terms of SRs that satisfy yet another condition. For example, in Hall's theorem \cite{hall} the extra condition is injectivity. In Rado's theorem \cite{rado} the condition is injectivity, plus the demand that the range of the choice function belongs to a given matroid. In the most general setting, a simplicial complex (closed down hypergraph) $\cc$ is given on $\bigcup F_i$, and the condition is that the range of the function belongs to $\cc$. The SR is then called a $\cc$-{\em SR}. (Even injectivity can be formulated in the terminology of $\cc$-SRs, using a trick of making many copies of each element.)

Given a graph $G$, we denote by $\ci(G)$ the set of independent sets in $G$.  An
$\ci(G)$-SR is also called an {\em ISR} (independent system of representatives).

We shall focus our attention on one special case, in which the sets $F_i$ are hypergraphs, and the extra condition is that the edges chosen are disjoint. In this case, the system of representatives is called a {\em rainbow matching} for $\cf$.  Thus, a rainbow matching for sets of edges $F_1,\ldots ,F_m$
is an ISR in the line graph of $\bigcup F_i$.

A  {\em partial $\cc$-SR}  is a partial choice function satisfying the above condition.

The conjecture that motivates our inquiry is the following conjecture of Aharoni and Berger (not published before):

\begin{conjecture}\label{fullhalf}\label{conj:main}
$k$ matchings of size $k+1$ in a bipartite graph possess a (full) rainbow matching.
\end{conjecture}

This conjecture strengthens a conjecture of Brualdi and Stein, on Latin squares.
A  Latin square of order $n$ is an $n \times n$ matrix whose entries are the symbols $1,\ldots ,n$,
each  appearing once in every row and once in every column. A (partial) {\em transversal}
in a Latin square is a set of entries, each in a distinct row and a distinct column, containing distinct symbols. A transversal of size $n$ is said to be {\em full}. A well known conjecture of Ryser \cite{ryser} is that for $n$ odd every $n \times n$ Latin square contains a full transversal. For even $n$ this is false, and Brualdi
and Stein \cite{brualdi,stein} raised  independently the following natural conjecture:

\begin{conjecture}\label{conj:bs}
In a Latin square of order $n$ there exists a partial transversal of size $n-1$.
\end{conjecture}

A Latin square can be viewed as a $3$-partite hypergraph, with sides
($R$=set of rows, $C$=set of columns, $S$=set of symbols), in which
every entry $e$ corresponds to the edge $(row(e), column(e),
symbol(e))$. The $3$-partite hypergraph corresponding to a Latin square satisfies stringent conditions - every  pair of vertices in two different sides belongs to precisely one $3$-edge. In particular, this means that for every vertex $v$, the set of $2$-edges complementing $v$ to a $3$-edge is a matching in a bipartite graph. Hence the Brualdi-Stein conjecture would follow from:

\begin{conjecture}\label{emptyhalf}
$k$ matchings of size $k$ in a bipartite graph possess a partial rainbow matching of size $k-1$.

\end{conjecture}

Conjecture \ref{emptyhalf} follows from Conjecture \ref{fullhalf} by the familiar device of
 expanding the  given matchings of size $k$ to
 matchings of size $k+1$,  adding the same edge to all.

 Note that $k$ matching of size $k$ need not have a full rainbow matching. This  is shown by a  standard example:-

\begin{example}\label{standard}

$$F_1=\ldots=F_{k-1}=\{(a_1,b_1),(a_2,b_2),\ldots ,(a_k,b_k)\},~~F_k=\{(a_1,b_2),(a_2,b_3),\ldots ,(a_k,b_1)\}$$\end{example}

Ian Wanless (private communication) constructed the following example for $k\ge 4$ even, in which all matchings $F_i$ are of size $k$, except for one which is of size $k+1$, and yet there is no rainbow matching:

\begin{example}[Wanless]
Write $k=2m$. Let $F_1, F_2,\ldots ,F_m$ all be equal to the matching $\{(a_i,b_i) \mid i \le k\}$, let $F_{m+1}, F_{m+1},\ldots ,F_{2m}$ all be the matching $\{(a_i,b_{i+1}) \mid i < k\}$ (where indices are taken $\bmod (k-1)$), with the edge $(a_k,b_k)$ added to all $F_i,~i <2m=k$ and the two edges  $(a_k,b_{k+1}),(a_{k+1},b_k)$ added to $F_k$.

\end{example}

  We do not know such an example for odd $k$, or an example of $k$ matchings, among which two  are of size $k+1$ and the rest of size $k$,  not possessing a rainbow matching.

Later on, many ramifications of Conjecture \ref{fullhalf} will be mentioned. But we shall start with a natural endeavor -
trying to prove as small lower bounds as possible on the size of the matchings, that guarantee the existence of a rainbow matching.

\begin{definition}
Let $f(r,k)$ be the  least number $t$ such that every $k$ matchings
of size $t$ in an $r$-partite graph have a rainbow matching. Also
let $g(r,k)$ be the largest number $s$ such that every $k$ matchings
of size $k$ in an $r$-partite bipartite graph possess a partial
rainbow matching of size $s$.
\end{definition}

In this terminology, Conjecture \ref{emptyhalf} is that $g(2,k)\ge k-1$, and Conjecture \ref{fullhalf} is that $f(2,k)\le k+1$.
Example \ref{standard} shows that:
\begin{observation}
For $k>1$ we have $g(2,k)\leq k-1$ and $f(2,k)\geq k+1$.
\end{observation}

Greedy arguments yield $f(2,k) \le 2k-1$ and $g(2,k) \ge \frac{k}{2}$.
Woolbright \cite{woolbright}  proved (though in a somewhat different context):

\begin{theorem}
$g(2,k) \ge k-\sqrt{k}$.
\end{theorem}

Another simple fact, already noted above in the case $k-g(r,k)=1$:
\begin{observation}
$f(r,k)-k \ge k-g(r,k)$
\end{observation}

\begin{proof}
Write $p$ for $f(r,k)-k$. Let $F_1, \ldots ,F_k$ be $k$ matchings of size $k$, in an $r$-partite hypergraph, we want to prove the existence of a partial rainbow matching of size $k-p$. Let $Q=\{e_1,\ldots ,e_p\}$ be
a matching, whose edges  are disjoint from  $\bigcup_{i \le k} F_i$, and let $F'_i=F_i \cup Q$.  By the definition of $f(r,k)$,
the matchings $F'_i$ have a (full) rainbow matching, and removing the edges belonging to $Q$ yields a partial rainbow matching of size at least $k-p$.
\end{proof}

The following observation shows that $k-g(r,k)$ is not bounded by a constant:

\begin{observation}\label{obs:gsmall}
 $g(r,2^{r-1})\le 2^{r-2}$.
\end{observation}

The proof  uses:

\begin{lemma}\label{twor-1}
For every $r>1$ there exists a system  of $2^{r-1}$ matchings
of size $2$ in an $r$-partite hypergraph,  not possessing a rainbow
matching.
\end{lemma}

\begin{proof}
Let $a_i,b_i$ be distinct elements, $1 \le i \le r$, and for every subset $T$ of $[r]$ let
$e_T=\{a_i: i \in T\}\cup \{b_i: i \not \in T\}$ and $f_T=\{a_i: i \not \in T\}\cup \{b_i: i  \in T\}$, and let $M_T=\{e_T,f_T\}$. Since $M_T=M_{[r]\setminus T}$, there are $2^{r-1}$ such matchings, and clearly they do not possess a rainbow matching.
\end{proof}

\begin{proof} (of Observation \ref{obs:gsmall}):
 Take $2^{r-2}$ disjoint copies $S_i$ of the above construction, and
 let $N_i$ be the set of $2^{r-1}$ matchings of size $2$ in  $S_i$ from that construction.
Decompose $\bigcup N_i$ into $k=2^{r-1}$  matchings $M_i$  of size
$k$, each consisting of one pair $e_T,f_T$ from each $S_i$.  The
largest partial rainbow matching of the matchings $M_i$ is of size
$2^{r-2}$, obtained by choosing one edge from each $S_i$.
\end{proof}

We do not know of any examples refuting $g(r,k) \ge k-2^{r-2}$ or $f(r,k) \le k+2^{r-2}$ (if true, this would fit in with Conjecture \ref{fullhalf}).

In the next two sections we shall prove the following two theorems:

\begin{theorem}\label{sevenfourths}
$f(2,k) \le \frac{7}{4}k$.
\end{theorem}

\begin{theorem}\label{tripartitehalf}
$g(3,k) \ge \frac{1}{2}k$.
\end{theorem}


\section {Proof of Theorem \ref{sevenfourths}}\label{sec:proof1}
Let $G$ be a bipartite graph with sides $A$ and $B$, and let $F_{i},
i=1,\ldots, k$ be matchings of size $n$ in $G$, with
 $n\geq 7k/4$.
We have to show that they possess a rainbow matching. We assume, for contradiction,
that this is not the case. We also apply an inductive hypothesis, by which we may assume that the matchings $F_2, \ldots ,F_k$ have a rainbow matching
$M=\{f_{2},f_{3},\ldots,f_{k}\}$, where $f_{i}\in F_{i}$.

By the assumption that $\cf$ does not possess a rainbow matching, every edge in $F_1$ meets some vertex of $\bigcup M$. Denote by $F'_{1}$ the set of edges of $F_{1}$ that are incident with exactly one vertex in $\bigcup M$. Since $F_{1}$ is a matching of size $n$, there are at least $(n-k+1)$ vertices in $A \setminus \bigcup M$  that are incident with an edge of $F'_{1}$. Similarly, there are at least $(n-k+1)$ vertices in $B \setminus \bigcup M$  that are incident with an edge of $F'_{1}$.
 Hence at least $(n-k+1)$ edges in $M$ meet an edge in $F'_{1}$, and since $n>7k/4$ implies $2(n-k+1)>k$, we have that there must be some edge in $M$ incident with two edges in $F'_{1}$. Without loss of generality, assume that $f_{2}$ is one of these edges and let $f'_{1}$ and $f''_{1}$ be the two edges in $F'_{1}$ meeting $f_2$. We may also assume that the edges of
$M$ incident with at least one edge of $F'_{1}$ belong, respectively, to
$F_2,F_3,\ldots,F_t$ (where, as shown above,  $t> n-k$).

%
%

Now choose, if possible, two edges $f_2',f_2'' \in F_2$ satisfying:

\begin{enumerate}
\item
both $f_2'$ and $f_2''$ meet
 $f_{i_3}$ for some $2<i_3 \le t$.

 \item
 both $f_2'$ and $f_2''$ do not meet
 any other
 edge from $M$ or any of the edges $f'_1,f_1''$.
 \end{enumerate}

 Without loss of generality,
we may assume that $i_3=3$.  Then choose, if
 possible, two edges $f_3',f_3'' \in F_3$ incident with
 $f_{i_4}$ for some $3<i_4 \le t$,  such that both do not meet any
 other edge from $M$ or any of the edges $f'_j,f_j'',~j <3$. Without loss of generality,
we may assume that $i_4=4$. Continuing this way until we reach a stage $p$ in which a choice as above is impossible, we obtain a sequence of edges $f_j',f_j''$ for $1 \le j <p$,
both meeting $f_{j+1}$, but not meeting any other edge of $M$ or any other
$f_i'$ or $f_i''$.

%
%

Write $M_1=\{f_2, \ldots ,f_p\}$, and let $P_1$ be the set of
vertices $\bigcup_{1 \le j < p} (f_j' \cup f_j'')$ (note that $P_{1}$ contains all vertices from $\bigcup M_{1}$). Let
$M_2=\{f_{p+1},\ldots, f_t\},~M_3=\{f_{t+1},\ldots ,f_k\}$ and let
$P_2=\bigcup M_2,~ P_3=\bigcup M_3$. We have:  $|P_1|=4(p-1),~ |P_2|=2(t-p), |P_3|=2(k-p)$. Let $S =V(G) \setminus (P_1 \cup P_2 \cup P_3)$.

\begin{figure}[h!]
 \begin{center}
\end{center}
\label{figproof74}
\end{figure}

\begin{claim} There are at most $t-p$ edges of $F_{p}$ joining a vertex  $P_{2}$ with a vertex of $S$.
\end{claim}
\begin{proof}
Since the process of choosing edges $f_j',~f_j''$ terminated at
$j=p$, there do not exist $g,h \in F_{p}$ incident with $S$ and
incident with the same $f \in M_2$. Since $M_{2}$ contains $t-p$
edges, this proves the claim.
\end{proof}

\begin{claim}
There are no edges of $F_{p}$ between $S$ and $P_1$ or inside $S$.
\end{claim}
\begin{proof}
If such an edge $f$ existed,  it would start an alternating path whose application to $M$ would result
in a rainbow matching for $F_1, \ldots ,F_k$: replace $f_p$ by $f$ as a representative for $F_p$; at least one of $f_{p-1}',f_{p-1}''$ does not meet $f$,
and this edge can replace $f_{p-1}$ as a representative of $F_{p-1}$, and so on..., until one of $f_1',f_1''$ can represent $F_1$.
\end{proof}

\begin{claim} \hspace{1cm}
\begin{enumerate}
 \item
 An edge $f \in F_p$ contained in $P_1$ must meet both $f_j'$ and $f_j''$ for some $j<p$.
 \item
  There exists at most
one index $j<p$ for which there exists an edge in $F_p \setminus M$ that meets $f_j'$ and $f_j''$.
\item At most $p$ edges of $F_{p}$ are contained in $P_1$ (these can be the $p-1$ edges $f_2, \ldots ,f_p$, plus one
edge connecting the non-$M$ vertices of some $f_j' \cup f_j''$).
\end{enumerate}

\end{claim}

\begin{proof}
Part (1) is proved as above - an edge not meeting $f_j'$ and $f_j''$ for any $j<p$ would start an alternating path whose
application would yield a full rainbow matching for $F_1, \ldots ,F_k$.

For the proof of part (2) of the claim, let $f$ be an edge in $F_p
\setminus M$ and let $j<p$ be such that $f$ meets $f_j'$ and
$f_j''$. Recall that by the definition of the choice of the edges
$f_i, i \le p$, we know that there exists
 $f_1 \in F_1$ that meets $f_p$.
  Then $f_1$ meets $f_j'$ or $f_j''$, or else $f_1, f_2, \ldots ,f_{p-1},f,f_{p+1},\ldots,f_k$
  would form a rainbow matching for $F_1, \ldots ,F_k$. Since $f_1$ can meet $f_j' \cup f_j''$ for only one $j$, this proves part (2) of the claim.

Part (3) follows from part (1) and part (2).
\end{proof}

Now we just have to count the number of possible edges in $F_{p}$. Denote by
\begin{itemize}
\item $t_{1}$ the number of edges of $F_{p}$ inside $P_1$.
\item $t_{2}$ the number of edges of $F_{p}$ between $P_1$ and $P_2$,
\item $t_{3}$ the number of edges of $F_{p}$ between $P_1$ and $P_3$,
\item $t_{4}$ the number of edges of $F_{p}$ inside $P_2$,
\item $t_{5}$ the number of edges of $F_{p}$ between $P_2$ and $P_3$,
\item $t_{6}$ the number of edges of $F_{p}$ between $P_2$ and $S$,
\item $t_{7}$ the number of edges of $F_{p}$ inside $P_3$,
\item $t_{8}$ the number of edges of $F_{p}$ between $P_3$ and $S$.
\end{itemize}

We then have the following relations, the first three following from the above claims, and the others from the fact that $F_{p}$ is a matching.
\begin{itemize}
\item $\sum_{i=1}^{8} t_{i}=|F_{p}|=n$
\item $t_{1}\leq p$
\item $t_{6}\leq t-p$
\item $2t_{1}+t_{2}+t_{3}\leq 4(p-1)$
\item $t_{2}+2t_{4}+t_{5}+t_{6}\leq 2(t-p)$
\item $t_{3}+t_{5}+2t_{7}+t_{8}\leq 2(k-t)$
\end{itemize}

Multiplying the second one by 1, the third one by 1, the fourth one by 1, the fifth one by 2, and the sixth one by 3 and adding them all gives :

$$3t_{1}+3t_{2}+4t_{3}+4t_{4}+5t_{5}+3t_{6}+6t_{7}+3t_{8}\leq p+(t-p)+4(p-1)+4(t-p)+6(k-t)=6k-t$$

Now we use $n=\sum t_{i}$ and $t>n-k$ to get the contradiction.
\begin{eqnarray*}
3n &\leq& 6k-t\\
3n &<& 6k-(n-k)\\
 n&<& 7k/4,
\end{eqnarray*}
\endpf


\section{Proof of Theorem \ref{tripartitehalf}}\label{sec:proof2}

Let $H=(V,E)$ be a 3-partite graph, and let $F_{i},~1\le i\le k$ be matchings of size $k$. We have to show that they possess a partial rainbow matching of size $k/2$.
Let $M$ be a maximum rainbow matching. Without loss of generality, assume that  $M=\{f_{1},\ldots,f_{p}\}$, where $f_{i}\in F_{i}$.
Let $i\leq p$ and $j>p$. We say that $f_{i}\in M$ is a good edge for  $F_j$ if there exists two distinct edges $f'_{j}$ and $f''_{j}$ in $F_{j}$ intersecting $f_{i}$ such that $|f'_{j}\cap \bigcup M|=|f''_{j}\cap \bigcup M|=1$.\\

\begin{claim}\label{kminus2p}
For any $j> p$, there are at least $(k-2p)$ good edges for  $F_j$.
\end{claim}

\begin{proof}
Since $M$ is maximal, every edge in $F_{j}$ is incident to at least one edge in $M$.
For $f\in M$ define
$$\phi(f)=\sum_{e\in F_{j}} \frac{|e \cap f|}{|e \cap \bigcup M|}$$
Clearly, $\phi(f)\leq 3$. In the sum defining $\phi(f)$ there can
occur the fractions
$\frac{1}{1}$,~$\frac{2}{2}$,~$\frac{1}{2}$,~$\frac{1}{3}$,~$\frac{2}{3}$,~$\frac{3}{3}$.
If $f$ is not a good edge then in the sum defining $\phi(f)$ there
can be at most one term $\frac{1}{1}$. Since the sum of the
numerators in the non-zero terms is at most $3$, this implies that
if $f$ is not good then
 $\phi(f)\leq 2$.

Note also that for each edge $e$ in $F_{j}$, we have that
$\sum_{f\in M} |f\cap e|=|e\cap \bigcup M|$. Therefore
$$
k=\sum_{e\in F_{j}}\frac{1}{|e\cap \bigcup M|}\sum_{f\in  M} |f\cap
e| =\sum_{f\in M}\phi(f)\leq 2(p-g)+3g=g+2p
$$
where $g$ denotes the number of good edges, and this gives the desired inequality.
\end{proof}

\begin{claim} \label{no3} No edge in $M$ is  good  for 3 distinct matchings not represented in $M$.
\end{claim}

\begin{proof} Denote by $A,B,C$ the three sides of the hypergraph. In the following
a vertex denoted by $a_{i}$ (resp $b_{i}$,$c_{i}$) will always belong to $A$ (resp. $B$,$C$).
Moreover  $a_{i}$ and $a_{j}$ for distinct $i$ and $j$ will always denote distinct vertices of the graph.
Assume by contradiction that such an edge $e$ exists. Its vertices are $(a_{1},b_{1},c_{1})$, and
it is a good edge for three distinct indices $j_{1},j_{2},j_{3}$.

Therefore for $s=1,2,3$, there exist in $F_{j_{s}}$ two edges
$e^{s}$ and $f^{s}$ meeting $e$ in $exactly$ one vertex. Note that
for $s\neq s'$ it is not possible that $e^{s}$ and $e^{s'}$ are
disjoint. Indeed in that case one could replace $e$ in the matching
by these two edges, contradicting the maximality of $M$. Of course
this is also true for $e^{s}$ and $f^{s'}$, so that amongst these 6
edges all pairs intersect except for the pairs $(e^{s},f^{s})$
(which are disjoint since they come from the same matching).

Without loss of generality we can assume that
$e^{1}=(a_{1},b_{2},c_{2})$ and $f^{1}=(a_{2},b_{1},c_{3})$. Again
without loss of generality we can assume that $e^{2}$ meets $e$ in
$a_{1}$. Since $e^2$ has to intersect $f^{1}$, and since $e^2$ does
not contain $b_{1}$, this implies that $e^{2}$ contains $c_{3}$.

Consider now $e^{3}$ and $f^{3}$ : if one of these edges contains
$b_{1}$, since it cannot contain $a_{1}$, it will fail to meet both
$e^{1}$ and $e^{2}$. Without loss of generality we can therefore
assume that $e^{3}$ contains $a_{1}$ and $f^{3}$ contains $c_{1}$.
But then as before, since $e^{3}$ meets $f^{1}$, it has to contain
$c_{3}$. But now $f^{2}$ is subject to the same constraints as
$e^{3}$ and $f^{3}$ just before, it cannot contain $b_{1}$ or else
it will fail to intersect $e^{1}$ and $e^{3}$. Hence $f^{2}$
contains $c_{1}$.

Now we have that $e^{2}$ and $e^{3}$ both contains $a_{1}$ and $c_{3}$ and $f^{2}$ and $f^{3}$ both contain $c_{1}$. But since $f^{2}$ and $f^{3}$ must intersect $e^{1}$ it implies that both need to contain $b_{2}$.  But now we get a contradiction because $e^{2}$ cannot contain $b_{2}$ and since $f^{3}$ cannot contain $a_{1}$, these two edges do not meet.

\end{proof}

By Claims \ref{kminus2p} and \ref{no3} we have:

$$2p \ge \sum_{j>p}|\{e \in M: e~ \text{is good for}~ F_j\}|\ge (k-p)(k-2p)$$

namely
$$2p^{2}-(3k+2)p+k^{2}\leq 0$$
which in turn implies that $p$ is larger than the smallest root of the quadratic expression:

$$p\geq \frac{3k+2-\sqrt{(3k+2)^{2}-8k^{2}}}{4} > \frac{3k+2-\sqrt{(k+6)^{2}}}{4} =\frac{k}{2}-1 $$
\endpf



\section{Putting the main conjecture in context}\label{conjectures}

\subsection{An observation and three offshoots}
Hypergraph matching theory abounds with conjectures and is meager with results.
 In such a field  putting order to the conjectures is of value. The aim of this section is to  place Conjecture \ref{fullhalf} in a general setting, and relate it to other conjectures, some known and some new.
While Conjecture \ref{fullhalf} generalizes the Brualdi-Ryser-Stein conjecture, its most natural background is probably the following observation, proved by a greedy  argument:

\begin{observation}\label{trivial}
Any set  $\cf=(F_1,\ldots ,F_k)$  of independent sets in a matroid $\cm$, where $|F_i|=k$ for all $i$, has an $\cm$-SR.
\end{observation}

 As often happens, when a fact is true for a very simple reason, it can be strengthened, and acquire depth by the addition of  other  ingredients.
 In this particular case, we are aware of three possible such ingredients:

 \begin{enumerate}
 \item Adding another matroid, namely replacing $\cm$ by the intersection of two matroids.

 \item Decomposability, meaning requiring the existence of ``many'' $\cm$-SR's, in the sense that $\bigcup \cf$ is the union of $k$ $\cm$-SR's.

     \item A ``scrambled'' version, obtained by scrambling the $F_i$'s, resulting
      in another family of $k$ sets of size $k$, for which an $\cm$-SR is sought. More generally, we can consider a general family of sets $F_1, \ldots ,F_m$, where $m$ is arbitrary, and the assumption that all $F_i$ are in $\cm$  can be replaced by an assumption that $\bigcup_{i \le m}F_i$, taken as a multiset, can be decomposed into  $k$ independent sets. (In the ``scrambled'' version $m=k$.)

        \end{enumerate}

The effect of adding each of these ingredients is different. Case  (1) of Observation \ref{trivial} is the subject of Conjecture \ref{conj:main}. So, the observation as is becomes false, needing strengthening of its condition in order to be possibly true. In the case of (2) Observation \ref{trivial} becomes a famous conjecture of Rota. And when adding ingredient (3) it still remains true.

But then things can become even more complicated, when two of the ingredients are  added together, or even all three.

\subsection{Adding another matroid}

Adding another matroid renders Observation \ref{trivial} false. As noted, a special case is that of rainbow matchings in bipartite graphs, and as we know a price of $1$ has to be paid there, namely the matchings have to be of size $k+1$. Here is the general, matroidal conjecture:

\begin{conjecture}\label{matroidalfullhalf}
Let $\cm$ and $\cn$ be two matroids on the same vertex set. If $F_1,\ldots ,F_k$ are sets of size $k+1$  belonging to $\cm \cap \cn$, then they have an $\cm \cap \cn$-SR. \end{conjecture}

\subsection{Decomposition problems}
Adding a requirement for decomposability into rainbow matchings results in a conjecture of Rota:

\begin{conjecture}\label{rota}
Given a set  $\cf=(F_1,\ldots ,F_k)$  of  independent sets in a matroid $\cm$, where $|F_i|=k$ for all $i$,  the multiset union $\bigcup \cf$ can be decomposed into $k$ $\cm$-SR's.
\end{conjecture}

\subsection{Scrambling}
 Scrambling the elements of the sets $F_i$, namely re-distributing them among the members of another family of $k$ sets, each member being of size $k$, retains the validity of Observation
 \ref{trivial}.
 In fact, the number of sets $F_i$ in $\cf=(F_1, \ldots ,F_m)$ can be arbitrary, in which case the ``scrambling'' terminology is no longer appropriate, and should be replaced by  the condition  that $\bigcup_{i \le m}F_i$ is decomposable into $k$ independent sets. To refer to this situation, we shall use the following terminology:

\begin{definition}
The {\em chromatic number} of a hypergraph $\cc$, denoted by $\chi(\cc)$, is the minimal number of edges of $\cc$ needed to cover $V(\cc)$.
\end{definition}

This parameter is also sometimes denoted in the literature by  $\rho(\cc)$. The name ``chromatic number'' comes from the fact that when $\cc$ is the complex of independent sets in a graph $G$, it is just the chromatic number of $G$. The fractional counterpart $\chi^*(\cc)$ is the minimal sum of weights on edges from $\cc$, such that every vertex belongs to edges whose sum of weights is at least $1$.

In this terminology, the modified version of Observation \ref{trivial} is:

         \begin{theorem}
         If $\cf=(F_1,\ldots ,F_m)$  is a set of (not necessarily independent) sets of size $k$ in a matroid $\cm$, and if $\chi(\cm)\le k$ (meaning that $\bigcup \cf$ is the union of $k$ independent sets), then $\cf$ has an $\cm$-SR. \end{theorem}

This follows directly from Rado's theorem \cite{rado}.

\subsection{Adding a matroid, and at the same time bounding the chromatic number}

Here is a theorem that explains the  title of this subsection:

\begin{theorem}\label{thm:mcapnsr}
If $\cm,~\cn$ are matroids on the same vertex set satisfying $\chi(\cm \cap \cn) \le k$, and $\cf=(F_1, \ldots ,F_m)$ is a family of sets belonging to $\cm \cap \cn$,  all of size $2k$, then $\cf$ has a $\cm \cap \cn$-SR.

\end{theorem}

This is a corollary of results from \cite{matcomp}. We shall not prove it here, and instead relate to the case
 in which $\cm$ and $\cn$ are partition matroids, so that $\cm \cap \cn$ is the complex of matchings in a bipartite graph $G$. By K\"onig's edge coloring theorem, the condition $\chi(\cm \cap \cn) \le k$ is equivalent to $\Delta(G) \le k$. Theorem \ref{thm:mcapnsr} is then a special case of the following theorem:

\begin{theorem}\label{abm}
Let  $\cf=\{F_1,\ldots ,F_m\}$ be a set  of   $q$-uniform hypergraphs on the same vertex set. If $|F_i| \ge
q\Delta(\bigcup_{i \le m}F_i)$ for all $i$ (here the union is taken as a multiset, namely degrees are counted with multiplicity), then $\cf$ has a rainbow matching.
\end{theorem}

This is an immediate corollary of:
\begin{theorem}\label{abmcorollary}\cite{abm}
Let  $\cf=\{F_1,\ldots ,F_m\}$ be a set  of   $q$-uniform hypergraphs on the same vertex set. If
$\nu^*(\bigcup_{i \in I}F_i) >q(|I|-1)$ for each $I \subseteq [m]$ then $\cf$ has a rainbow matching.
\end{theorem}

If the condition of Theorem \ref{abm} holds, then the constant fractional matching $f(e)=\frac{1}{\Delta(\bigcup_{i \in I}F_i)}$ is of total weight
at least $q|I|$, and hence the condition of Theorem \ref{abmcorollary} holds.

Theorem \ref{abm} is tight even when $F_i$ are $q$-partite, but in the only example we know that shows tightness the hypergraphs $F_i$ are multihypergraphs, meaning that they contain repeated edges. The following example is taken from \cite{absz}:

\begin{example}\label{example:absz} For $i=1, \ldots ,k$ let $F_i$ be a matching $M_i$ of size $q$, repeated $k$ times (each edge of $M_i$ is of size $q$). Let $F_{k+1}$ consist of
$k$ matchings $N_i$, each of size $q$, such that each edge in $N_i$ meets each edge in $M_i$. Then $|F_i|=kq$ for all $i \le k+1$, the degree of every vertex
in $\bigcup F_i$ is $kq+1$, and there is no rainbow matching.
\end{example}

It is of interest to understand whether the repeated edges are essential for this example.

Write $c(r)$ for the minimal number
 for which there exists a number $d(r)$, such that any  $r$-partite hypergraph with sides $V_1,V_2, \ldots ,V_r$ satisfying
$\delta(V_1) \ge c(r) {\Delta(V_2 \cup \ldots \cup V_r)+d(r)}$ has a matching covering $V_1$. By Theorem \ref{abm}
$c(r) \le r-1$.
For $r$ such that there exists a projective plane of edge size $r-1$ it is possible to show that $c(r) \ge r-2$.

\begin{conjecture}
$c(r)=r-2$ for all $r>2$.
\end{conjecture}

\subsection{Combining all three ingredients}

Let us return to Theorem \ref{abmcorollary}, and re-formulate it in terms of so called {\em bipartite hypergraphs}. In a bipartite hypergraph there is a {\em special side}, call it $S$, such that every edge intersects $S$ at precisely one vertex.  Let $r=q+1$, and form
a bipartite $r$-uniform hypergraph by assigning a vertex $v_i \in M$ to each set $F_i$ and   forming an $r$-tuple   $(v_i,e)$
for every edge $e \in F_i$. In this terminology Theorem \ref{abmcorollary} is:

\begin{theorem}\label{thm:bipartite}
If in a bipartite $r$-uniform hypergraph and special side $S$ it is true that $deg(u)\ge (r-1)deg(v)$ for every $u \in S$ and $v \in V \setminus S$, then there exists a matching of $S$ (namely, covering $S$).

\end{theorem}
It may well be that the conclusion can be strengthened, to the effect that there is a partition of $E(H)$ into $\Delta(H)$ such matchings:

\begin{conjecture}\cite{abw}\label{conj:decomposingtoisrs}

Let $H$ be a hypergraph satisfying the conditions of Theorem \ref{thm:bipartite}. Then $\chi'(H)=\Delta(H)$.
\end{conjecture}

Here $\chi'(H)$ denotes the edge chromatic number of $H$. For $r=3$ this is a generalization of a conjecture of Hilton \cite{hilton}:

\begin{conjecture}

 An $n \times 2n$ Latin rectangle can  be decomposed into $2n$ transversals.
  \end{conjecture}

  In \cite{hagg-joh} Hilton's conjecture was proved for $n \times (1+o(1))n$ Latin rectangles.

In \cite{abw} Conjecture \ref{conj:decomposingtoisrs} was proved for $|S|=2$. Another result there was half of what is required: the conjecture is true under the stronger condition  $deg(u)\ge 2(r-1) deg(v)$ for every $u \in S$ and $v \in V \setminus S$.

\subsection{Scrambling and decomposing together - a scrambled Rota conjecture}

What happens in Rota's conjecture if we first scramble the elements? That is, if the sets $F_i,~i=1, \ldots,k$ are not necessarily bases, but $\bigcup_{i \le k}F_i$ is the (multiset) union of $k$ bases? In \cite{ak} it was shown that for $k$ odd there does not necessarily exist a decomposition into $k$ $\cm$-SRs.
We do not know a counterexample in the even case. Moreover, the following may be true:

\begin{conjecture}\cite{ak}
If $F_1, \ldots ,F_k$ are sets of size $k$ in a matroid $\cm$, such that $\chi(\cm)\le k$, then there exist $k-1$ disjoint $\cm$-SRs.
\end{conjecture}

In fact, we can also ask this question  for a general number of sets. here we exert some measure of caution, and pose it in the form of a question, rather than a conjecture:

\begin{prob}
If $F_1, \ldots ,F_m$ are sets of size $k$ in a matroid $\cm$ satisfying $\chi(\cm)\le k$, do there necessarily exist $m-1$ disjoint $\cm$-SRs?
\end{prob}



\subsection{Decompositions in the intersection of two matroids}
Here is yet another conjecture stemming from the combination of all three ingredients. We are dealing with two matroids, assuming something about their chromatic numbers, and require a low chromatic number of the intersection:

\begin{conjecture}\label{abconj}\cite{ab}
For any pair of matroids $\cm,~\cn$ on the same ground set, $$\chi(\cm \cap \cn) \le \max(\chi(\cm), \chi(\cn))+1.$$
\end{conjecture}

A special case relates to ``scrambled Rota'', in which the matroids are the original matroid and the partition matroid defined by the given sets. Here is the conjecture, explicitly:

\begin{conjecture}
Given sets $F_1, \ldots ,F_m$ of size $k$ in a matroid $\cm$ satisfying $\chi(\cm) \le k$, there exist $k+1$ $\cm$-SRs whose union is $\bigcup_{i \le m}F_i$.
\end{conjecture}

  Conjecture \ref{abconj} is close in spirit to a well known conjecture of  Goldberg and Seymour \cite{goldberg, seymour}:

\begin{conjecture}\label{conj:gs}
In any multigraph $\chi' \le \chi'^*+1$.
\end{conjecture}

 The kinship between the two conjectures can be given a precise formulation: in \cite{ab} a common generalization of the two was suggested, in terms of $2$-polymatroids.

 As often happens in this field, ``half'' of Conjecture \ref{abconj} is known to be true:

\begin{theorem}\cite{ab}\label{matroidintersectioncoloring}
For  $\cm,~\cn$ as above,  $\chi(\cm \cap \cn) \le 2\max(\chi(\cm), \chi(\cn))$.
\end{theorem}

\begin{corollary}{\hfill}

If $F_1, \ldots ,F_m$ are independent sets of size $k$ in a matroid $\cm$ satisfying $\chi(\cm) \le k$  then there exist $2k$ $\cm$-SRs whose union is $\bigcup_{i \le m}F_i$.
\end{corollary}

\subsection{Distinct edges}

It is an intriguing fact that in some theorems and conjectures on hypergraph matchings the only known examples showing sharpness use repeated edges, or even repeated sets $F_i$. It is  tempting to conjecture that under an assumption of distinctness (of either edges or sets) the conditions can be weakened.  Example \ref{example:absz} is one case in point. Here are two more examples:

\begin{theorem}\label{drisko}\cite{drisko, ab}

$2k-1$ matchings of size $k$ in a bipartite graph have a partial rainbow matching of size $k$.

\end{theorem}

The example showing sharpness is $F_1=\ldots=F_{k-1}=\{(a_1,b_1),(a_2,b_2),\ldots ,(a_k,b_k)\}$, and  $F_k=F_{k+1}=\ldots=F_{2k-2}=\{(a_1,b_2),(a_2,b_3),\ldots ,(a_k,b_1)\}$. It is tempting to make the following conjecture:

\begin{conjecture}\label{conj:disjoint}

$k+1$ disjoint matchings of size $k$ in a bipartite graph have a partial rainbow matching of size $k$.

\end{conjecture}

Here is yet another  generalization of the Ryser-Brualdi-Stein conjecture:

\begin{conjecture}
In a $d$-regular $n \times n \times n$ $3$-partite simple hypergraph
(not containing repeated edges) there exists a matching of size at
least $\lceil \frac{d-1}{d}n \rceil$.
\end{conjecture}

For $d=2$ the conjecture is true, by the inequality $\nu \ge
\frac{1}{2}\tau$ proved in \cite{ryser3} (in a regular $n\times
n\times n$ $3$-partite hypergraph $\tau=n$, because
$\tau^*=\nu^*=n$). The Ryser-Brualdi-Stein conjecture is obtained by
taking $d=n$.  If true, the
conjecture is sharp for all $n$ and $d$. To see this, write
$n=kd+\ell$, where $\ell<d$, and take $k$ disjoint copies of a
$3$-partite $d$-regular $d \times d \times d$ hypergraph with
$\nu=d-1$, together with a disjoint  $\ell \times \ell \times
\ell$~~$d$-regular $3$-partite hypergraph. To see that the non-repetition of edges is essential, take the Fano plane with a vertex deleted
(containing $4$ edges: $(a_1,b_1,c_1),(a_1,b_2,c_2),(a_2,b_1,c_2),(a_2,b_2,c_1)$), and repeat every edge $d/2$ times for any $d$ even. Then  the degree of every vertex is $d$, and the matching number is $1$.

{\bf Acknowledgements:} The authors are indebted to Ian Wanless and Janos Barat for helpful discussions. Part of the research of the first author was done while visiting them at Monash University.

\end{document}